\documentclass[a4paper,11pt]{article}

\usepackage{hyperref}

\usepackage{amsmath,amsthm,amssymb}

\usepackage[dvips]{geometry}
\usepackage{graphicx}
\usepackage[noadjust]{cite}

\usepackage[capitalize]{cleveref}

\newcommand{\diff}{\,\mathrm{d}}

\newcommand{\dt}{\diff t}
\newcommand{\dX}{\diff X}
\newcommand{\dW}{\diff W}

\newcommand{\R}{\mathbb{R}}
\newcommand{\Ih}{I^{\hmax}}
\newcommand{\hmax}{{\tilde{h}}}
\newtheorem{definition}{Definition}
\newtheorem{theorem}{Theorem}

\newtheorem{example}{Example}

\providecommand{\E}{\operatorname{E}}
\providecommand{\N}{\mathbb{N}}
\newcommand{\xvar}{{x}}
\newcommand{\bO}{\ensuremath{\mathcal{O}}}

\newcommand{\rh}[1]{\rho(#1)}
\newcommand{\varphit}{{\varphi_{{t}}}}
\newcommand{\Mpure}{\mu}
\newcommand{\dMpure}{\diff\Mpure}
\newcommand{\M}{{\Mpure_{{t}}}}

\newcommand{\numweightpure}{\Phi}
\newcommand{\numweight}{{\numweightpure^{\mu}_{{t}}}}
\newcommand{\ODEnumweight}{{\numweightpure_{{t}}}}
\newcommand{\hnumweight}{\hat{\numweightpure}}

\newcommand{\pd}{\ensuremath{p_d}}
\newcommand{\pmu}{p_\Mpure}
\newcommand{\DM}[1]{{\Delta_{{#1}}\Mpure}}

\begin{document}

\title{High order numerical integrators for single integrand Stratonovich SDEs}

\author{{David Cohen$^{1,2}$\thanks{e-mail: david.cohen@chalmers.se}},
	\ Kristian Debrabant$^3$\thanks{e-mail: debrabant@imada.sdu.dk}
	\ and Andreas R\"o{\ss}ler$^4$\thanks{e-mail: roessler@math.uni-luebeck.de}
	\bigskip
	\\
	\small{$^1$Department of Mathematical Sciences, } \\
	\small{Chalmers University of Technology and University of Gothenburg, Gothenburg, Sweden} \\[0.2cm]
	\small{$^2$Department of Mathematics and Mathematical Statistics, } \\
	\small{Umeå University, Sweden} \\[0.2cm]
	\small{$^3$Department of Mathematics and Computer Science,} \\
	\small{University of Southern Denmark, Denmark} \\[0.2cm]
	\small{$^4$Institute of Mathematics,} \\
	\small{Universit\"at zu L\"ubeck, Germany}
}

\date{}

\maketitle

\begin{abstract}
We show that applying any deterministic B-series method of order $p_d$ with a random step size
to single integrand SDEs gives a numerical method converging in the mean-square and weak sense
with order $\lfloor p_d/2\rfloor$.
As an application, we derive high order energy-preserving methods for stochastic Poisson systems
as well as further geometric numerical schemes for this wide class of Stratonovich SDEs.
\end{abstract}

\section{Introduction}\label{sec:intro}
The last years have seen a great interest in the numerical analysis of single integrand
Stratonovich stochastic differential equations
\begin{equation}
  \dX(t) =  f(X(t))\left(\lambda\dt + \sigma\circ \dW(t))= f(X(t)\right)\circ \dMpure(t), \qquad {X}(t_0)=x_0,
  \label{eq:sde}
\end{equation}
see, e.\,g.\ \cite{MR629977,MR1747753,MR1951908,MR3506778,MR2674236,MR2861707,MR2956993,MR3218332,
cohen14epi,MR3506248,debrabant17cah,MR3973463,MR3952248} as well as to the text below.
Here, $t\geq t_0 \geq 0$, $x_0\in\R^d$, $f\colon\R^d\to\R^d$,
$(W(t))_{t \geq 0}$ is a standard one-dimensional Wiener process,
$\lambda\in\{0,1\}$ and $\sigma \in \R$ are given constants, and
$\Mpure({s}):=\lambda{s} + \sigma W(s)$ for $s\geq 0$.

Stochastic differential equations (SDEs) of this form arise, e.\,g., when the right-hand side of
an ordinary differential equation (ODE) model is randomly perturbed. Interesting applications are:

\begin{example}[Fatigue cracking \cite{sobczyk87smf}]
  \[
  \dX(t)=a X^p(t)\dt+b X^p(t)\circ\dW(t),\qquad p>1, a,b\in\R.
  \]
\end{example}

\begin{example}[Stochastic Hamiltonian systems \cite{MR3579605,MR1951908}]
\[\dX(t)
  =J^{-1} \nabla H(X(t)) \left(\dt+c\circ\dW(t)\right),\qquad J=\begin{pmatrix} 0 & -I \\ I & 0\end{pmatrix}, c\in\R.
\]
\end{example}

\begin{example}[Stochastic perturbations of Poisson systems \cite{cohen14epi}]
\[\dX(t)
  =B(X(t)) \nabla H(X(t)) \left(\dt+c\circ\dW(t)\right),\qquad B=-B^\top, c\in\R.
\]
\end{example}

In the present communication, we generalize the main results from \cite{debrabant17cah} from Runge--Kutta methods
to B-series methods. In particular, this permits to show that
when applying any (deterministic) B-series numerical integrators of order $p_d$
to \eqref{eq:sde} by replacing the deterministic step size $h$ in the numerical scheme
by $\lambda h+\sigma (W(t+h)-W(t))$ when calculating one step of the approximation on the
time interval $[t,t+h]$ respectively,
one then automatically
obtains a time integrator with step size $h$ of mean-square as well as
weak order $\lfloor p_d/2\rfloor$ for the SDE \eqref{eq:sde}.
We recall that the notation $\lfloor x\rfloor$ denotes the floor function of a real number $x$.

On top of that, it can easily
be observed that, in general, when a (deterministic) numerical integrator possesses
some geometric properties, then the same geometric properties hold for the time integrator applied
to single integrand SDEs \eqref{eq:sde}. This observation thus allows to carry forward
various results already obtained in the literature as well as new ones, see for instance \cref{sec:rdmDE} below.

B-series methods encompass any reasonable one-step time integrators for ODEs:
Besides Runge--Kutta methods \cite{hairer10sodI,hairer06gni}, other one-step methods can be represented as B-series,
e.\,g.\ Taylor methods \cite{barrio05pot,debrabant11bao,debrabant11cos},
averaged vector field methods \cite{quispel08anc,celledoni09epr}, and $q$-derivative Runge--Kutta methods \cite{hairer10sodI}.

\section{Convergence of B-series methods applied to single integrand SDEs}\label{sec:B-seriesconvergence}
B-series for solutions to SDEs and their numerical solution by stochastic Runge--Kutta methods have been developed in
\cite{burrage96hso,burrage00oco} to study strong convergence in the Stratonovich case,
in \cite{MR2739567,MR2669396} for strong converge in the It\^{o} as well as the Stratonovich case, in \cite{komori97rta} and
\cite{komori07mrt} to study weak convergence in the Stratonovich case and in \cite{roessler04ste,roessler06rta} to study weak convergence in both the
It\^{o} and the Stratonovich case. A uniform and self-contained theory for the construction of stochastic B-series for the exact solution of SDEs
and its numerical approximation by stochastic Runge--Kutta methods is given in \cite{debrabant08bao}. Based on the notation used there, in \cite{debrabant17cah} the
B-series for the exact solution and Runge--Kutta approximations of single-integrand SDEs were derived.
For convenience, we summarize the results we will need in the following.
Due to the single integrand we do, similar to the ODE case \cite{butcher08nmf}, only need non-colored trees in the expansion of the solution.
\begin{definition}[Trees]
  The set of rooted trees $T$ related to single-integrand SDEs is recursively defined as follows:
  \begin{enumerate}
    \item[a)] The empty tree $\emptyset$ and the graph $\bullet=[\emptyset]$ with only one vertex {belong} to $T$.
  \end{enumerate}
  For a positive integer $\kappa$, let $\tau=[\tau_1,\tau_2,{\dots},\tau_{\kappa}]$ be the tree
  formed by joining the subtrees
  $\tau_1,\tau_2,{\dots},\tau_{\kappa}$ each by a single branch to a
  common root.
  \begin{enumerate}
    \item[b)] If $\tau_1,\tau_2,{\dots},\tau_{\kappa} \in T$ then
      $\tau=[\tau_1,\tau_2,{\dots},\tau_{\kappa}] \in T$.
  \end{enumerate}
\end{definition}

\begin{definition}[Elementary differentials]
  For a tree $\tau \in T$ the elementary differential is
  a mapping $F(\tau):\R^d \rightarrow \R^d$ defined
  recursively by
  \begin{enumerate}
    \item[a)] $F(\emptyset)(\xvar)=\xvar$, \\ \mbox{}
    \item[b)] $F(\bullet)(\xvar)=f(\xvar)$, \\ \mbox{}
    \item[c)] If $\tau=[\tau_1,\tau_2,{\dots},\tau_{\kappa}] \in T\setminus\{\emptyset\}$
      then
      \[
	F(\tau)(\xvar)=f^{(\kappa)}(\xvar)
      \big(F(\tau_1)(\xvar),F(\tau_2)(\xvar),{\dots},F(\tau_{\kappa})(\xvar)\big)\]
  \end{enumerate}
{where $\xvar\in\R^d$.}
\end{definition}

\begin{definition}[B-series]
  Consider for each $\tau\in T$ a stochastic process $(\phi(\tau)(h))_{h\geq0}$ satisfying
  \[
    \phi(\emptyset)\equiv1 \;\text{ and }\; \phi(\tau)(0)=0,\quad
    \text{for } \tau\in T \backslash \{\emptyset\}.
  \]
  Let $\xvar\in\R^d$ and $h\in\R$. A (stochastic) B-series is then
  a formal series of the form
  \[
    B(\phi,\xvar; h) = \sum_{\tau \in T}
    \alpha(\tau)\cdot\phi(\tau)(h)\cdot F(\tau)(\xvar),
  \]
  where
  $\alpha: T\rightarrow \mathbb{Q}$ is given by
  \begin{align*}
    \alpha(\emptyset)&=1,&\alpha(\bullet)&=1,
    &\alpha(\tau=[\tau_1,\cdots,\tau_{\kappa}])&=
    \frac{1}{r_1!r_2!\cdots r_{q}! } \prod_{j=1}^{\kappa} \alpha(\tau_j),
  \end{align*}
  where $r_1,r_2,{\dots},r_{q}$ count equal trees among
  $\tau_1,\tau_2,{\dots},\tau_{\kappa}$.
\end{definition}

We are now able to state the B-series of the exact solution to the SDE \eqref{eq:sde},
see also \cite{debrabant08bao,MR2739567}.
In the following,
$\rh{\tau}$ denotes the number of nodes in a tree $\tau$.

\begin{theorem}[\cite{debrabant17cah}]\label{th:BseriesexactsolutionODE}
  Let
  $\gamma: T\rightarrow\N$ be given by
  \begin{gather*}
    \gamma(\emptyset)=1,\qquad\gamma(\bullet)=1,\\
    \gamma([\tau_1,\dots,\tau_{\kappa}])=\rh{[\tau_1,\dots,\tau_{\kappa}]}\prod_{j=1}^{\kappa} \gamma(\tau_j).
  \end{gather*}
  Then the solution $X({t+}h)$ of \eqref{eq:sde} {starting at the
point $(t,\xvar)$} can be written as a B-series $B(\varphit,\xvar; h)$ with
  \begin{gather}\label{eq:explweightexsol}
    \varphit(\tau)(h)=\frac{\M(h)^{\rh{\tau}}}{\gamma(\tau)}\qquad\text{ for }\tau\in T
  \end{gather}
  where
  \begin{equation}\label{eq:defMut}
\M(s):=\Mpure(t+s)-\Mpure(t)=\lambda s + \sigma (W(t+{s}){-W(t)}).
\end{equation}
\end{theorem}
That is, the B-series of the exact solution of \eqref{eq:sde} coincides with the
 one of the exact solution of an ODE (i.\,e.\ when $\lambda=1$, $\sigma=0$) when replacing $\M(h)$ with $h$.

Next, we derive the B-series for the numerical solution of the single integrand SDE \eqref{eq:sde}.
To do this, we consider B-series methods for the ODE $\dX(t) =  f(X(t))\dt$. Such methods can
be written as \cite{hairer06gni}
\begin{equation}\label{eq:BseriesODE}
Y(t+h)=B(\ODEnumweight,x;h)
\end{equation}
 with ${\ODEnumweight}(\tau)(h)=h^{\rho(\tau)} \cdot \hnumweight(\tau)$ where
$\hnumweight:T\to\R$ fulfilling $\hnumweight(\emptyset)=1$.
Replacing $h$ by $\DM{t,t+h}=\M(h)$ defined in \eqref{eq:defMut}
we consider therefore the following B-series method for solving \eqref{eq:sde}:
\begin{definition}\label{def:ODEBSeriesmethod}
Given an ODE-B-series method \eqref{eq:BseriesODE} for the ODE $\dX(t) =  f(X(t))\dt$, the corresponding ODE-B-series method for \eqref{eq:sde} is given by
\begin{equation}\label{eq:Bseriesmethod}
  Y(t+h)=B(\numweight,x;h),
\end{equation}
where $\numweight(\tau){(h)}={(}\DM{{t,t+h}}{)}^{\rh{\tau}}\hnumweight(\tau)$.
\end{definition}
Now we give the definitions of both weak and strong convergence used in this note.

Let $C_P^l(\R^d, \R^{\hat{d}})$ denote the space of all $g \in
C^l(\R^d,\R^{\hat{d}})$ fulfilling a polynomial growth condition \cite{kloeden99nso}
and $\Ih$ be the discretized time interval on which the numerical approximations are calculated.
\begin{definition}
  A time discrete approximation $Y=(Y(t))_{t \in {\Ih}}$ converges weakly with order $p$ to $X$ at time $t \in {\Ih}$ as {the maximum step size}
  $\hmax\rightarrow 0$ if for each $g \in
  C_P^{2(p+1)}(\R^d, \R)$ there exist a constant $C_g$
  and a finite $\delta_0 > 0$ such that
  \begin{equation*}
    | \E(g(Y(t))) - \E(g(X(t))) | \leq C_g \, \hmax^p
  \end{equation*}
holds for each $\hmax\in \, ]0,\delta_0[\,$.
\end{definition}
Whereas weak approximation methods are used to estimate the expectation of functionals of the solution, strong approximation methods approach the solution
path-wise.
In this article, next to weak convergence, we will consider mean-square instead of strong convergence.
\begin{definition}
  A time discrete approximation $Y=(Y(t))_{t \in {\Ih}}$ converges
  in the mean-square sense with order $p$ to $X$ at time $t \in {\Ih}$ as
  {the maximum step size} $\hmax\rightarrow 0$ if there {exist} a constant
  $C$
  and a finite $\delta_0 > 0$ such that
  \begin{equation*}
    \sqrt{\E(\|Y(t) - X(t)\|^2)}\leq C \, \hmax^p
  \end{equation*}
holds for each $\hmax \in \, ]0,\delta_0[\,$.
\end{definition}
Observe that, by Jensen's inequality, mean-square convergence implies strong convergence of the same order.

We are now in position to state the main result of the present publication.
\begin{theorem}\label{th:main}
Assume that the B-series method \eqref{eq:Bseriesmethod} is of deterministic order $\pd$ and let $\pmu = \lfloor\pd/2 \rfloor$.
Further, let $f\in C^{2\pmu+1}(\R^d,\R^d)$
and $f$ and $f'f$ fulfill a Lipschitz condition.
 Finally, assume
\begin{itemize}
    \item for mean-square convergence, that all elementary differentials $F(\tau)$ fulfill a linear growth condition,
\item respectively for weak convergence, that
    $f\in C_P^{2\pmu+1}(\R^d,
        \R^d)$.
  \end{itemize}
Then this very same B-series method is of mean-square as well as weak order $\pmu$ when applied to the single integrand SDE \eqref{eq:sde}
with step size $\DM{t,t+h}$.
  For weak convergence, it suffices that $\DM{t,t+h}$ is chosen such that at least the first $2\pmu+1$ moments coincide
with those of $\M(h)$, and all the others are in $\bO(h^{\pmu+1})$.
\end{theorem}
\begin{proof}
To prove this result, one first observes that the corresponding Theorem for Runge--Kutta methods \cite[Theorem 1.1]{debrabant17cah} only uses
that the B-series of the exact and numerical solutions fulfil the properties summarized in \cref{def:ODEBSeriesmethod}.
One can then directly extend the proof to the present situation of B-series methods.
\end{proof}

\section{Applications to geometric numerical integration of single integrand SDEs}\label{sec:rdmDE}
Geometric properties of numerical schemes for single integrand SDEs \eqref{eq:sde} are closely related
to those of numerical schemes for their corresponding ODEs. This relation, which holds for B-series as well
as non B-series (e.\,g.\ volume preserving) methods, can be specified furthermore: In principle, when applying
any deterministic geometric numerical integrators to single integrand SDEs \eqref{eq:sde} with random step size $\DM{t,t+h}$,
one then obtains the same geometric property as by the corresponding deterministic numerical scheme.
This is because the random perturbation in some sense respects the geometric structure of the phase space.
Hence, proofs of conservation properties in the deterministic setting can be adapted to single integrand Stratonovich SDEs \eqref{eq:sde}.
With the main idea used in this paper, one easily derives several results already obtained in the literature
as well as new ones for geometric numerical integrators of single integrand SDEs \eqref{eq:sde}:
\begin{enumerate}
\item If the deterministic numerical method preserves linear or quadratic invariants, so does it for single integrand SDEs.
See the example below and e.g.\  \cite{MR3273263}.
\item If the deterministic scheme is energy or invariant-preserving, then it is also energy or invariant-preserving for single integrand SDEs.
See the example below and e.g.\  \cite{MR1747753,cohen14epi,MR3973463,MR3952248}.
\item If the deterministic numerical method is symmetric, then it is also symmetric for the SDE \eqref{eq:sde}.
\item If the deterministic time integrator is symplectic and $f(x)=J^{-1}\nabla H(x)$, then it is symplectic for single integrand SDEs, e.g.\  \cite{MR2674236,MR1951908,MR3579605}.
\item Deterministic (symmetric) projection methods can be directly applied to single integrand SDEs with constraints, e.g.\  \cite{WANG2019123305}.
\item If the deterministic numerical method is volume preserving, then it is also volume preserving for single integrand SDEs.
\item If the deterministic numerical scheme is a Poisson integrator and $f(x)=B(x)\nabla H(x)$, where $B(x)$ represents a Poisson bracket, then it is also a Poisson integrator for single integrand SDEs.
\end{enumerate}

\section{Application to stochastic perturbations of Poisson systems}\label{sec:appli}
As application of Theorem~\ref{th:main} and the principles of the previous section,
we propose a high order energy-preserving and Casimir-preserving scheme
for the stochastic rigid body \cite{cohen14epi}. The equations of motion of this stochastic rigid body
are a Lie-Poisson system:
\begin{eqnarray}\label{srb}
\begin{pmatrix}
\dX_{[1]}\\
\dX_{[2]}\\
\dX_{[3]}
\end{pmatrix}
=
\begin{pmatrix}
 0 & -X_{[3]} & X_{[2]}\\
X_{[3]} & 0 & -X_{[1]}\\
-X_{[2]} & X_{[1]} & 0
\end{pmatrix}
\begin{pmatrix}
X_{[1]}/I_1\\
X_{[2]}/I_2\\
X_{[3]}/I_3
\end{pmatrix}
\bigl(\dt+c\circ\dW\bigr),
\end{eqnarray}
with some initial value $X(0) \in \mathbb{R}^3$
at $t_0=0$, where $X=(X_{[1]},X_{[2]},X_{[3]})^\top$ and
$I=(I_1,I_2,I_3)$ are the moments of inertia.
The Hamiltonian
\[
H(X)=\frac12\big(X_{[1]}^2/I_1+X_{[2]}^2/I_2+X_{[3]}^2/I_3\big)
\]
is thus a conserved quantity as well as the quadratic Casimir
\[
C(X)=\|X\|_2^2=X_{[1]}^2+X_{[2]}^2+X_{[3]}^2.
\]
Note that the right hand side $f$ of \eqref{srb} is not globally Lipschitz continuous and does thus
a priori not fulfill the conditions of Theorem~\ref{th:main}.
However, for methods that conserve the Casimir, we can replace $f$ by a function being zero outside a suitable ball and fulfilling the conditions of Theorem~\ref{th:main}, e.\,g.\ replacing $f$ by
\[
\tilde{f}(X)=f(X)\frac{\varphi(4C(X(0))-C(X))}{\varphi(C(X)-2C(X(0)))+\varphi(4C(X(0))-C(X))}
\]
where
\[\varphi(r)=\begin{cases}
  0&\text{for }{r\leq0}\\
  e^{-1/r}&\text{for }r>0.
\end{cases}
\]
Note that $\tilde{f}(X)=f(X)$ for all $\|X\|_2\leq\sqrt{2}\|X(0)\|$ and $\tilde{f}(X)=0$ for all $\|X\|_2\geq2\|X(0)\|$ (see e.\,g.\ \cite{MR1930277}), and all derivatives of $\tilde{f}$ remain bounded. A similar modification can be made for numerical methods that preserve the Hamiltonian of the deterministic rigid body.
\newline
As a starting deterministic method, we choose the linear energy-preserving integrators of orders
$2,4$ and $6$ from \cite{MR2784654}.
By Theorem~\ref{th:main}, we thus expect a first, resp.\ second, resp.\ third strong and weak order
energy-preserving numerical integrator for the stochastic rigid body \eqref{srb}, denoted by EPs1, EPs2, EPs3.
Recall that these novel numerical methods are simply given by replacing the step size
$h$ by $\lambda h+\sigma\Delta W_n$ in the deterministic scheme.
The orders of convergence can be seen in Figures~\ref{fig:srb} and \ref{fig:wrb},
where the numerically determined orders of convergence $\hat{p}$ correspond to the slopes of the dashed regression lines.
The parameters
for these numerical experiments are: $T_{end}=0.5, c=0.5, I=(0.345,0.653,1), X(0)=(0.8,0.6,0)$,
and $200$ samples are used to approximate the expectations. We have checked that $200$ samples are sufficient for these numerical experiments.
The reference solution is calculated with $h_{ref}=2^{-14}$ and EPs3.
\begin{figure}
\begin{center}
\includegraphics*{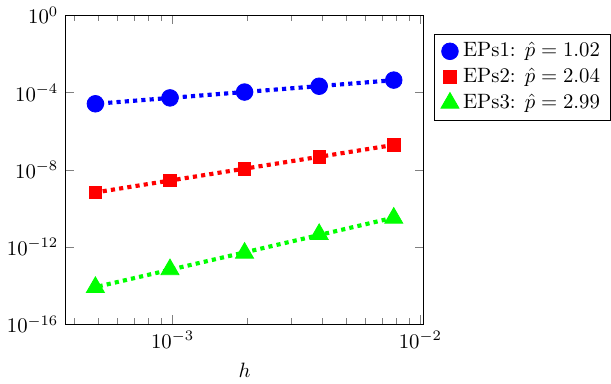}
\caption{Mean-square orders of convergence of the numerical schemes EPs1 to EPs3 for the stochastic rigid body \eqref{srb}.
Used step sizes: $h=2^{-7}$ to $h=2^{-11}$.}
\label{fig:srb}
\end{center}
\end{figure}

\begin{figure}
\begin{center}
\includegraphics*{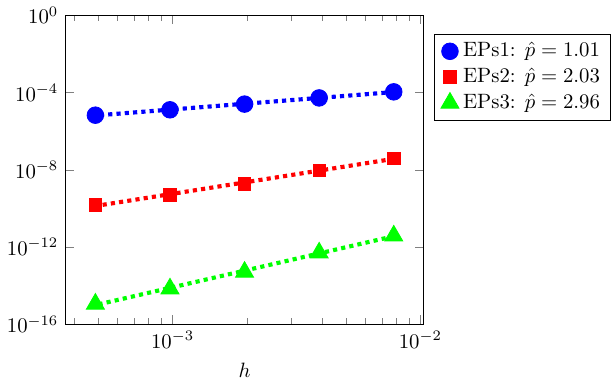}
\caption{Weak orders of convergence of the numerical schemes EPs1 to EPs3 for the second moments
of approximations of solutions to the stochastic rigid body \eqref{srb}. Used step sizes: $h=2^{-7}$ to $h=2^{-11}$.}
\label{fig:wrb}
\end{center}
\end{figure}
Furthermore, as it can be seen in Figure~\ref{fig:invariants}, for the example of the third order method,
these numerical schemes not only preserve the energy but also the quadratic Casimir $C(X)$. This follows from the discussion in Section~\ref{sec:rdmDE}, as the numerical integrators from \cite{MR2784654} are known to preserve quadratic Casimirs.
\begin{figure}
\begin{center}
\includegraphics*[height=6cm,keepaspectratio]{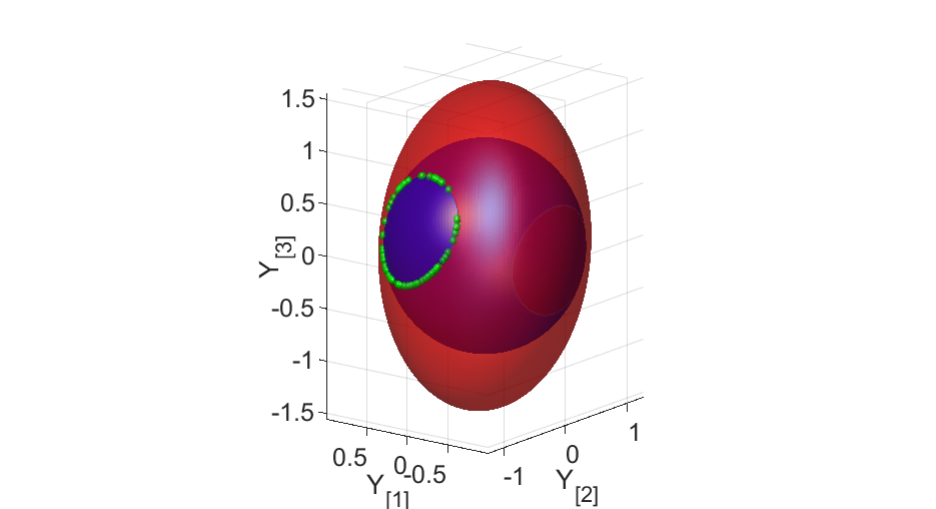}
\caption{Numerical simulation with the third order method (one path) for the stochastic rigid body problem (green points) and the two invariants $H(Y)=H(Y(0))$ (blue) and $C(Y)=C(Y(0))$ (red), with $T_{end}=5$, $h=2^{-4}$, other parameters as before.}
\label{fig:invariants}
\end{center}
\end{figure}

\section*{Acknowledgements}
We appreciate the referees' comments on an earlier version of the paper.
The work of DC was partially supported by the Swedish Research Council (VR) (projects nr. 2018-04443).

\end{document}